\documentclass[12pt,a4paper]{article}
\usepackage{theorem}
\newtheorem{lemma}{Lemma}
\newtheorem{proposition}[lemma]{Proposition}
\newtheorem{theorem}[lemma]{Theorem}
\newtheorem{corollary}[lemma]{Corollary}
\newtheorem{exampl}[lemma]{Example}
\newenvironment{example}{\begin{exampl}\upshape}{\hfill$\Box$\end{exampl}}

\newcommand{\R}{{\bf R}}

\newcommand{\rme}{{\rm e}}
\newcommand{\rmd}{{\rm d}}

\newcommand{\cB}{{\cal B}}

\newcommand{\cE}{{\cal E}}

\newcommand{\cV}{{\cal V}}
\newcommand{\sig}{\sigma}
\newcommand{\alp}{\alpha}
\newcommand{\bet}{\beta}

\newcommand{\lam}{\lambda}

\newcommand{\eps}{\varepsilon}

\newcommand{\tr}{{\rm tr}}
\newcommand{\Dom}{{\rm Dom}}
\newcommand{\Quad}{{\rm Quad}}

\newcommand{\norm}{\Vert}

\newcommand{\Schrodinger}{Schr\"odinger }

\newcommand{\pr}{\prime}

\newcommand{\eqref}[1]{(\ref{#1})}

\newenvironment{proof}{\textbf{Proof}}{\hfill$\Box$}

\newenvironment{choices}{ \left\{ \begin{array}{ll} }{\end{array}\right.}

\usepackage{amssymb}
\usepackage{hyperref}
\usepackage{graphicx}
\title{An Inverse Spectral Theorem}
\author{E B Davies}
\date{7 September 2010}
\begin{document}
\maketitle
\begin{abstract}
We prove a substantial extension of an inverse spectral theorem of
Ambarzumyan, and show that it can be applied to arbitrary compact
Riemannian manifolds, compact quantum graphs and finite
combinatorial graphs, subject to the imposition of Neumann (or
Kirchhoff) boundary conditions.
\end{abstract}


\section{Introduction}\label{intro}

Let $X$ be a compact metric space provided with a finite measure
$\rmd x$. Let $H_0$ be a non-negative self-adjoint operator acting
on $L^2(X,\rmd x)$. We assume that $H_0$ has discrete spectrum $\{
\lam_n\}_{n=1}^\infty$, where the eigenvalues are written in
increasing order and repeated according to multiplicity, and that
its smallest eigenvalue $\lam_1=0$ has multiplicity $1$ with
corresponding eigenfunction $\phi_1=|X|^{-1/2}$, where $|X|$ is the
volume of $X$. Given a bounded real potential $V$ on $X$, we put
$H=H_0+V$, so that $H$ also has discrete spectrum, which we denote by $\{
\mu_n\}_{n=1}^\infty$. The problem is to
write down a general list of abstract conditions on $H_0$ which
imply that if $H$ and $H_0$ have the same spectrum, taking
multiplicities into account, then $V$ is identically zero.

The classical theorem of Ambarzumyan solved this problem when
$X=[a,b]$ and $H_0f=-\frac{\rmd^2f}{\rmd x^2}$, subject to Neumann
boundary conditions at $a$ and $b$, \cite{ambar}. The result is also
known for periodic boundary conditions, but the corresponding result
for Dirichlet boundary conditions is false; the best known inverse spectral theorem in this context depends on knowing the spectrum of $H$ for two different sets of boundary conditions at
$a,\, b$, \cite{borg}. Ambarzumyan's theorem has
been extended to trees with a finite number of edges,
\cite{CP,piv,LY}, by combining the Sturm-Liouville theory with a
careful boundary value analysis. The present paper extends it to a
much broader context by adapting a range of classical techniques
from the theory of the heat equation in several dimensions. Theorem~\ref{maincorollary}, establishes that
If $\mu_1\geq 0$ and $\limsup_{n\to\infty}(\mu_n-\lam_n)\leq 0$ then
$V=0$, subject to certain generic conditions on the heat kernels involved.

This theorem can be applied to arbitrary compact Riemannian manifolds, compact quantum graphs and finite combinatorial graphs, subject to Neumann (or Kirchhoff) boundary conditions. Our proof depends on a list of abstract hypotheses that are known to be satisfied in a wide variety of situations. The hypotheses are by no means the weakest possible; the strategy of our
proof is more important than the detailed assumptions, and can be
adapted to other cases.

The material in Sections~\ref{H0} and \ref{H} is of a general
character, and the reader may prefer to start in Section~\ref{MR}.
In Section~\ref{QG} we prove that all of the hypotheses hold for a
finite connected quantum graph $X$, subject to Kirchhoff boundary
conditions at every vertex.

Before proceeding, I should like to thank Professor Chun-Kong Law for a very stimulating lecture in the Isaac Newton Institute in July 2010, where the author learned about this problem.

\section{Properties of $H_0$}\label{H0}

We start by listing the hypotheses that will be used in the proofs.

\begin{description}
\item[(H1)] The operator $\rme^{-H_0t}$ has a non-negative integral kernel $K_0(t,x,y)$ for $t>0$, which is continuous on $(0,\infty)\times X\times X$.
\item[(H2)] There exist constants $c>0$ and $d>0$ such that
$0\leq K_0(t,x,x)\leq ct^{-d/2}$ for all $t\in (0,1)$.
\item[(H3)] There exists a constant $a>0$ such that
$\lim_{t\to 0}t^{d/2}K_0(t,x,x)=a$ for all $x\notin N$, where $N$ is
a set of zero measure.
\item[(H4)] The smallest eigenvalue $\lam_1$ of the operator $H_0$ equals $0$ and has multiplicity $1$. The corresponding eigenfunction is $\phi_1=|X|^{-1/2}$.
\end{description}

We do not assume that $H_0$ is a second order elliptic differential
operator, because we wish to allow other possibilities. For example
$H_0$ could be a fractional power of a Laplacian. The case in which
$H_0$ is a discrete Laplacian on $l^2(X)$ for some finite set $X$ is
discussed in Example~\ref{finite}. The conditions (H1) to (H4) have
been examined in some detail in \cite{HKST}, from which we quote the
following consequences of (H1) and (H4).

The quadratic form defined on $\Quad(H_0)=\Dom(H_0^{1/2})$ by
\[
Q_0(f)=\langle H_0^{1/2}f,H_0^{1/2}f\rangle
\]
is a Dirichlet form; see \cite[Theorem~1.3.2]{HKST}. The
one-parameter semigroup $T_t=\rme^{-H_0t}$ on $L^2(X,\rmd x)$ is an
irreducible symmetric Markov semigroup. It extends to a
one-parameter contraction semigroup on $L^p(X,\rmd x)$ for all
$1\leq p\leq \infty$, with the proviso that for $p=\infty$ the
semigroup is not strongly continuous; see \cite[Prop.~1.4.3]{HKST}.
Mercer's theorem, \cite[Prop.~5.6.9]{LOTS}, implies that the
operator $\rme^{-H_0t}$ is trace class for all $t>0$ and
\[
\tr[\rme^{-H_0t}]=\int_X K_0(t,x,x) \, \rmd x.
\]
In particular
\[
\sum_{n=1}^\infty \rme^{-\lam_n t}<\infty
\]
for all $t>0$, where $\{ \lam_n\}_{n=1}^\infty$ are the eigenvalues
of $H_0$ written in increasing order and repeated according to
multiplicities. If $\phi_n$ are the corresponding normalized
eigenfunctions then by applying the formula
\[
\rme^{-\lam_nt}\phi_n(x)=\int_X K_0(t,x,y)\phi_n(y) \, \rmd y
\]
we deduce that every eigenfunction $\phi_n$ is bounded and
continuous on $X$. The semigroup $T_t$ is ultracontractive in the
sense of \cite[Section 2.1]{HKST} and the series
\[
K_0(t,x,y)=\sum_{n=1}^\infty \rme^{-\lam_n t} \phi_n(x)\phi_n(y)
\]
converges uniformly on $[\alp,\infty)\times X\times X$ for every
$\alp>0$; see \cite[Theorem 2.1.4]{HKST}. This implies that
$K_0(t,x,y)$ converges uniformly to $|X|^{-1}$ on $X\times X$ as
$t\to\infty$, so
\begin{equation}
\frac{1}{2|X|}\leq K_0(t,x,y)\leq \frac{3}{2|X|}\label{lowerupper}
\end{equation}
for all large enough $t>0$.

The condition (H2) is much more specific, but necessary and
sufficient conditions for its validity are now classical.

\begin{proposition}\label{logsob} Let $H$ be a self-adjoint operator acting in $L^2(X,\rmd x)$. If $H$ is bounded below and $\rme^{-Ht}$ is positivity preserving for all $t\geq 0$ then the following are equivalent, the constant $d>0$ being the same in all cases.
\begin{enumerate}
\item The operator $\rme^{-Ht}$ satisfies
\[
\norm \rme^{-Ht}f\norm_\infty \leq c_1 t^{-d/4}\norm f\norm_2
\]
for some $c_1>0$, all $f\in L^2(X,\rmd x)$ and all $t\in (0,1)$.
\item The bound
\[
\int_X  f^2\log(f)\rmd x \leq \eps Q(f)+\bet(\eps)\norm f\norm_2^2
+\norm f\norm_2^2\log(\norm f\norm_2)
\]
holds for all $0\leq f\in \Quad(H)\cap L^1\cap L^\infty$ and all
$\eps \in (0,1)$, where $\bet(\eps)=c_2-(d/4)\log(\eps)$ for some
$c_2>0$. See \cite[Example~2.3.3]{HKST}.
\item The bound
\[
\norm f\norm_2^{2+4/d}\leq c_3\left( Q(f)+\norm
f\norm_2^2\right)\norm f\norm_1^{4/d}
\]
holds for some $c_3>0$ and all $0\leq f\in \Quad(H) \cap L^1$. See
\cite[Corollary~2.4.7]{HKST}.
\item Assuming $d>2$, the bound
\[
\norm f\norm_{2d/(d-2)}^{2}\leq c_4\left( Q(f)+\norm
f\norm_2^2\right)
\]
holds for some $c_4>0$ and all $f\in\Quad(H)$. See
\cite[Corollary~2.4.3]{HKST}.
\end{enumerate}
All of the above conditions imply that $\rme^{-Ht}$ has a measurable
heat kernel $K$ that satisfies
\begin{equation}
0\leq K(t,x,y) \leq c_5 t^{-d/2}\label{ptbound}
\end{equation}
for some $c_5>0$, almost all $x,y\in X$ and all $t\in (0,1)$.
Conversely, if $\norm \rme^{-Ht}\norm_{L^\infty\to L^\infty} \leq
c_6$ for all $t\in (0,1)$ then (\ref{ptbound}) implies the previous
conditions. See \cite[Lemma~2.1.2]{HKST}.

\end{proposition}

An important feature of all these conditions is that they depend on
the quadratic form $Q$ and can therefore be transferred from one
operator to another if the quadratic forms are comparable.

\begin{example}
Let $H_0=-\frac{\rmd^2}{\rmd x^2}$ act in $L^2((0,\infty), \rmd x)$
subject to Neumann boundary conditions at $0$. Then
\[
K_0(t,x,y)=(4\pi t)^{-1/2}\left(
\rme^{-(x-y)^2/(4t)}+\rme^{-(x+y)^2/(4t)}\right)
\]
so
\[
0\leq K_0(t,x,y)\leq 2 (4\pi t)^{-1/2}
\]
for all $t>0$ and $x,y\in (0,\infty)$. Moreover
\[
\lim_{t\to 0} (4\pi t)^{1/2}K_0(t,x,x)=\begin{choices}
1&\mbox{ if $x>0$,}\\
2&\mbox{ if $x=0$.}
\end{choices}
\]
This explains the need for an exceptional set of zero measure in
(H3).

One may also solve the corresponding example in
$\R^N_+=(0,\infty)^n$ subject to Neumann boundary conditions on the
boundary. In this case the possible values of $\lim_{t\to 0} (4\pi
t)^{1/2}K_0(t,x,x)$ are the integers $2^r$ where $0\leq r\leq n$. A related result for general convex sets is given in \cite[Theorem~12]{lipschitz}.
\end{example}

\section{Properties of $H=H_0+V$}\label{H}

Given a self-adjoint operator $H_0$ satisfying the hypotheses (H1)
to (H4), we put $H=H_0+V$ where $V$ is a bounded real-valued
potential; this condition can surely be weakened. An application of
the Trotter product formula or a perturbation expansion imply that
\[
\norm \rme^{-Ht}\norm _{L^p\to L^p} \leq \rme^{\norm V\norm_\infty
t}
\]
for all $p\in [1,\infty]$ and $t\geq 0$. By using standard
variational methods one sees that $H$ has discrete spectrum and that
its eigenvalues $\{\mu_n\}_{n=1}^\infty$, written in increasing
order and repeated according to multiplicity, satisfy
\[
\lam_n - \norm V\norm_\infty \leq \mu_n\leq \lam_n + \norm
V\norm_\infty
\]
for all $n\geq 1$. Hence
\[
0\leq \rme^{-\norm V\norm_\infty t} \tr[\rme^{-H_0 t}]\leq
\tr[\rme^{-H t}]\leq \rme^{\norm V\norm_\infty t}\tr[\rme^{-H_0 t}]
\]
for all $t>0$.

The proof of the following theorem involves standard ingredients,
\cite{chavel,HKST}, but we write it out in detail for the sake of
completeness.

\begin{theorem}\label{mult1}
The operator $\rme^{-Ht}$ has a non-negative continuous kernel $K$
for all $t>0$ and $x,y\in X$. The kernel satisfies
\begin{equation}
0\leq \rme^{-\norm V\norm_\infty t} K_0(t,x,y) \leq  K(t,x,y)\leq
\rme^{\norm V\norm_\infty t} K_0(t,x,y)\label{TP}
\end{equation}
for all $t>0$. The smallest eigenvalue $\mu_1$ of $H$ has
multiplicity $1$.
\end{theorem}

\begin{proof}
We will assume throughout the proof that $0<t<1$; once (\ref{TP})
has been proved in this case it can be extended to larger $t$ by
using the semigroup property. Since the quadratic form
\[
Q(f)=Q_0(f)+\int_X V(x)|f(x)|^2\, \rmd x
\]
is a Dirichlet form in the sense of \cite[Theorem~1.3.2]{HKST}, the
operators $\rme^{-Ht}$ are all positivity preserving. The quadratic
forms of $H_0$ and $H$ are comparable, so we may use
Proposition~\ref{logsob} to deduce that for every $t\in (0,1)$ there
is a bounded, measurable integral kernel $K(t,x,y)$ satisfying
$0\leq K(t,x,y)\leq ct^{-d/2}$ if $0<t<1$ and
\[
(\rme^{-Ht}f)(x)=\int_X K(t,x,y) f(y)\, \rmd y
\]
for all $f\in L^2$. Since \[
\rme^{-Ht}=\rme^{-H\eps}\rme^{-H(t-2\eps)}\rme^{-H\eps}
\]
for all $\eps >0$ and $t>2\eps$, we can use the norm analyticity of
$\rme^{-H(t-2\eps)}$ in $L^2$ to deduce the norm analyticity of
$\rme^{-Ht}$ from $L^1$ to $L^\infty$. This implies that
$K(t,\cdot,\cdot)$ depends analytically on $t$ in the
$L^\infty(X\times X)$ norm for $0<t<\infty$.

The upper and lower bounds in (\ref{TP}) are now direct applications
of the Trotter product formula. (\ref{lowerupper}) and (\ref{TP})
together imply that the operator $A=\rme^{-Ht}$ is irreducible for
all large enough $t>0$. Therefore its largest eigenvalue has
multiplicity $1$ by a direct application of
\cite[Theorem~13.3.6]{LOTS} to $A/\norm A\norm$.

The operator $\rme^{-Ht}$ has an operator norm convergent infinite
series expansion involving $\rme^{-H_0t}$ and $V$, but we will use
the more compact expression
\begin{equation}
\rme^{-Ht}=\rme^{-H_0t} -A(t)+B(t)\label{pertseries}
\end{equation}
where
\begin{eqnarray*}
A(t)&=&\int_{s=0}^t\rme^{-H_0(t-s)}V\rme^{-H_0s} \,\rmd s,\\
B(t)&=&\int_{s=0}^t\int_{u=0}^s\rme^{-H_0(t-s)}V%
\rme^{-H(s-u)}V \rme^{-H_0u} \,\rmd u\rmd s.
\end{eqnarray*}
The integrands are norm continuous in $\{ s:0<s<t\}$, resp.
$\{(s,u):0<u<s<t\}$, and they are uniformly bounded in norm, so the
integrals are norm convergent and $A(t),\, B(t)$ depend norm
continuously on $t$.

The equation (\ref{pertseries}) has a version involving integral
kernels, namely
\begin{equation}
K(t,x,y)=K_0(t,x,y)-L(t,x,y)+M(t,x,y)\label{expansion}
\end{equation}
for all $t>0$ and $x,y\in X$, where
\[
L(t,x,y)=\int_{s=0}^t\int_{z\in X} K_0(t-s,x,z)V(z)K_0(s,z,y)\, \rmd
z,
\]
and we will prove that
\begin{equation}
|M(t,x,y)|\leq ct^{2-d/2}\label{Morder}
\end{equation}
for all $t\in (0,1)$ and $x,y\in X$.

We will also prove that all the kernels on the right-hand side of
(\ref{expansion}) are continuous on $(0,1)\times X\times X$, and
this will establish that $K$ is continuous on the same set. The
estimates below involve the uniform norm $\norm \cdot\norm_\infty$
on $\cB=C(X\times X)$.

The integral kernel of $A(t)$ is
\begin{equation}
L(t,x,y)=\int_{s=0}^t L_{s,t}(x,y)\, \rmd s\label{int1}
\end{equation}
where $L_{s,t}:X\times X\to \R$ is defined by
\[
L_{s,t}(x,y)=\int_X K_0(t-s,x,z)V(z)K_0(s,z,y)\, \rmd z.
\]
Now $L_{s,t} \in \cB$ for all $0<s<t$ and $L_{s,t}$ depends norm
continuously on $s,t$ subject to these conditions. We have to prove
that the integral (\ref{int1}) is norm convergent in $\cB$. This
follows from
\begin{eqnarray*}
\int_{s=0}^t \norm L_{s,t}\norm_\infty \, \rmd s&\leq& \norm V\norm_\infty \int_{s=0}^t\sup_{x,y}\left\{\int_X K_0(t-s,x,z)K_0(s,z,y)\, \rmd z\right\} \rmd s\\
&=&\norm V\norm_\infty \int_{s=0}^t\sup_{x,y}\left\{K_0(t,x,y)\right\}\,\rmd s\\
&\leq & c\norm V\norm_\infty t^{1-d/2}
\end{eqnarray*}
provided $0<t<1$.

The integral kernel of $B(t)$ is
\begin{equation}
M(t,x,y)=\int_{s=0}^t\int_{u=0}^s M_{u,s,t}(x,y)\, \rmd u\rmd
s\label{int2}
\end{equation}
where $M_{u,s,t}:X\times X\to \R$ is defined by
\[
M_{u,s,t}(x,y)=\int_{X^2} K_0(t-s,x,z)V(z)K(s-u,z,w)V(w)K_0(u,w,y)\,
\rmd w\rmd z.
\]
Without assuming that $K(s-u,z,w)$ is continuous in $z, w$, one sees
by (\ref{TP}) that $M_{u,s,t} \in \cB$ for all $0<u<s<t$, and that
$M_{u,s,t}$ depends norm continuously on $u,s,t$ subject to these
conditions. We have to prove that the integral (\ref{int2}) is norm
convergent in $\cB$. We have
\begin{eqnarray*}
\norm M_{u,s,t}\norm_\infty &\leq& \norm V\norm_\infty^2 \norm
N_{u,s,t}\norm_\infty
\end{eqnarray*}
where
\begin{eqnarray*}
N_{u,s,t}(x,y)
&=&\int_{X^2} K_0(t-s,x,z)K(s-u,z,w)K_0(u,w,y)\, \rmd w\rmd z\\
&\leq &
\rme^{\norm V\norm_\infty t}\int_{X^2} K_0(t-s,x,z)K_0(s-u,z,w)K_0(u,w,y)\, \rmd w\rmd z\\
&= & \rme^{\norm V\norm_\infty t} K_0(t,x,y)\\
&\leq & \rme^{\norm V\norm_\infty } ct^{-d/2}.
\end{eqnarray*}
provided $0<t<1$. Therefore
\begin{eqnarray*}
\int_{s=0}^t\int_{u=0}^s\norm M_{u,s,t}\norm_\infty \, \rmd u\rmd s
&\leq& b t^{2-d/2}
\end{eqnarray*}
where $b$ depends on $\norm V\norm_\infty$.

\end{proof}

\begin{corollary}
One has
\begin{equation}
\tr[\rme^{-Ht}]=\tr[\rme^{-H_0t}] -t\int_X K_0(t,x,x)V(x)\, \rmd
x+\rho(t)\label{critical}
\end{equation}
where $\rho(t)=O(t^{2-d/2})$ as $t\to 0$.
\end{corollary}

\begin{proof}
One puts $x=y$ in (\ref{expansion}) and integrates with respect to
$x$. The bound on $\rho(t)$ follows from (\ref{Morder}).
\end{proof}

\section{The main results}\label{MR}

In this section we assume that $H_0$ satisfies (H1) to (H4) and that
$H=H_0+V$ where $V$ is a real bounded potential on $X$. Both
operators have discrete spectrum, and their eigenvalues are denoted
by $\{ \lam_n\}_{n=1}^\infty$, respectively $\{
\mu_n\}_{n=1}^\infty$, written in increasing order and repeated
according to multiplicity. We assumed that $\lam_1=0$ and proved
that $\mu_1$ has multiplicity $1$; see Theorem~\ref{mult1}. Our following theorem has something in common with \cite[Theorems~2.5, 3.4]{simon}, which obtain a related result for \Schrodinger operators in one and two dimensions subject to (\ref{intVleq0}).

\begin{theorem}\label{premain}
If $\mu_1\geq 0$ and
\begin{equation}
\int_X V(x)\, \rmd x\leq 0\label{intVleq0}
\end{equation}
then $V=0$.
\end{theorem}

\begin{proof}
The variational estimate
\[
\mu_1\leq  Q(\phi_1) = |X|^{-1} \int_X V(x)\, \rmd x \leq 0,
\]
where $\phi_1(x)=|X|^{-1/2}$, shows that $\mu_1=0$ under the stated
conditions. We apply the results of the last section to $H_s=H_0+sV$
where $s$ is a real parameter. The smallest eigenvalue
$F(s)=\mu_1(s)$ of $H_s$ has multiplicity $1$ for all $s\in\R$ and
therefore is an analytic function of $s$ by a standard argument in
perturbation theory. It is also concave by a variational argument.
Finally $F(0)=0$ and
\[
F^\pr(0)=\langle V\phi_1,\phi_1\rangle =|X|^{-1}\int_X V(x)\, \rmd x
\leq 0.
\]
Since $F(1)=0$, its concavity implies that $F(s)$ must equal $0$ for
all $s\in [0,1]$. By its analyticity, $F(s)=0$ for all $s\in\R$.

If $V$ does not vanish identically then (\ref{intVleq0}) implies
that its negative part cannot vanish identically. Therefore there
exists a function $\psi\in L^2(X,\rmd x)$ such that $\langle
V\psi,\psi\rangle <0$. An approximation argument allows us to assume
that $\psi\in \Quad(H_0)$. We now conclude that
\[
F(s)=Q_0(\psi) +s \langle V\psi,\psi\rangle <0
\]
for all large enough $s>0$. The contradiction implies that $V=0$.
\end{proof}

The following is our main inverse spectral theorem.

\begin{theorem}\label{maintheorem}
If $\mu_1\geq 0$ and $\limsup_{t\to 0} \sig(t)\leq 0$ where
\[
\sig(t)=t^{d/2-1}  \sum_{n=1}^\infty\left(
\rme^{-\lam_nt}-\rme^{-\mu_nt}\right)
\]
then $V=0$.
\end{theorem}

\begin{proof}
We rewrite (\ref{critical}) in the form
\begin{eqnarray*}
\lefteqn{t^{d/2} \int_X K_0(t,x,x)V(x)\, \rmd x}\\
&=& t^{d/2-1}\left\{ \tr[\rme^{-H_0t}]-\tr[\rme^{-Ht}]\right\}+t^{d/2-1}\rho(t)\\
&=& \sig(t)+t^{d/2-1}\rho(t)
\end{eqnarray*}
and then take the limit of both sides as $t\to 0$. The left hand
side converges to $a\int_X V(x)\, \rmd x$ where $a>0$, by (H2) and
(H3). We deduce that $\int_X V(x)\, \rmd x\leq 0$ and may therefore
apply Theorem~\ref{premain}.
\end{proof}

The following corollary of Theorem~\ref{maintheorem} contains the
original Ambarzumyan theorem as a special case.

\begin{theorem}\label{maincorollary}
If $\mu_1\geq 0$ and $\limsup_{n\to\infty}(\mu_n-\lam_n)\leq 0$ then
$V=0$.
\end{theorem}

\begin{proof}
Given $\eps >0$ there exists $N=N(\eps)$ such that $\mu_n-\lam_n\leq
\eps$ for all $n\geq N$.  We then have
\[
\sig(t)=\sig_1(t)+\sig_2(t)
\]
where
\begin{eqnarray*}
\sig_1(t)&=&t^{d/2-1}\sum_{n=1}^{N-1}\left( \rme^{-\lam_nt}-\rme^{-\mu_nt}\right)\\
&\leq & t^{d/2}\sum_{n=1}^{N-1} |\lam_n-\mu_n|,
\end{eqnarray*}
and
\begin{eqnarray*}
\sig_2(t)&=&t^{d/2-1}\sum_{n=N}^\infty\left( \rme^{-\lam_nt}-\rme^{-\mu_nt}\right)\\
&\leq & t^{d/2-1}\sum_{n=N}^\infty\left( \rme^{-\lam_nt}(1-\rme^{-\eps t})\right)\\
&\leq&\eps t^{d/2} \sum_{n=1}^\infty \rme^{-\lam_nt}\\
&\leq c\eps
\end{eqnarray*}
for all $t\in (0,1)$, by an application of (H2). We conclude that
$\limsup_{t\to 0} \sig(t)\leq c\eps$ for all $\eps
>0$, and may therefore apply Theorem~\ref{maintheorem}.
\end{proof}

\begin{example}\label{finite}
Let $H_0$ be a (non-negative) discrete Laplacian on $l^2(X)$ for
some finite, combinatorial graph $X$, with $|X|=n$. One can bypass
many of our calculations by using the elementary formula
\[
\sum_{r=1}^n\mu_r-\sum_{r=1}^n\lam_r=\tr[H-H_0]=\tr[V]=\sum_{x\in X}
V(x).
\]
The relevant conditions on the eigenvalues in this case are
\[
\mu_1\geq 0 \mbox{ and } \sum_{r=1}^n\mu_r\leq\sum_{r=1}^n\lam_r.
\]
However the analysis of the function $F$ in Theorem~\ref{premain}
requires the assumptions (H1) and (H4), and the use of the theory of
irreducible symmetric Markov semigroups.
\end{example}

\begin{theorem}\label{riem}
The hypotheses (H1) to (H4) and therefore the conclusions of
Theorems~\ref{maintheorem} and \ref{maincorollary} are valid if $H_0$ is the Laplace-Beltrami
operator on a compact, connected Riemannian manifold $X$, subject to
Neumann boundary conditions if $X$ has a boundary $\partial X$; the
boundary should satisfy the Lipschitz condition.
\end{theorem}

\begin{proof}
All of the hypotheses except (H2) and (H3) are minor variations on results in \cite{HKST}. the Lipschitz boundary condition is needed to obtain
(H2), $d$ being the dimension of $X$. This is a result of a general principle that bounded changes of the metric, and therefore of the local coordinate system, do not affect bounds such as (H2), \cite{lipschitz}. The precise heat kernel
asymptotics required in (H3) holds for all $x\notin \partial X$, and
is a small part of classical results of Minakshisundaram, Pleijel
and others concerning the small time asymptotics of the heat kernel;
see \cite{BGO,chavel,gilkey,MS}. In this context the nature of the
boundary is irrelevant by the principle of `not feeling the
boundary'; see Theorem~\ref{NFTB} and \cite{hsu}.
\end{proof}

\section{Finite quantum graphs}\label{QG}

In this section we prove that Theorems~\ref{maintheorem} and \ref{maincorollary} are
applicable when $X$ is a finite connected quantum graph. We assume
that $X$ is the union of a finite number of edges $e\in\cE$, each of
finite length. Each edge terminates at two vertices out of a finite
set $\cV$, and we assume that the graph as a whole is connected. The
operator $H_0$ acts in $L^2(X,\rmd x)$ by the formula
$H_0f(x)=-\frac{\rmd^2 f}{\rmd x^2}$, subject to Kirchhoff boundary
conditions at each vertex; more precisely we require that all
functions in the domain of $H_0$ are continuous and that the sum of
the outgoing derivatives vanishes at each vertex. All of our
calculations depend on the fact that the quadratic form associated
with $H_0$ is given by
\[
Q_0(f)=\int_X |f^\pr (x)|^2 \, \rmd x
\]
with domain $\Quad(H_0)=\Dom(H_0^{1/2})=W^{1,2}(X)$ where this is
the space of all functions $f$ whose restriction to any edge $e$
lies in $W^{1,2}(e)$, together with the requirement that $f$ is
continuous at every vertex. We observe that $W^{1,2}(X)$ is
continuously embedded in $C(X)$. It is immediate from its definition
that $Q_0$ is a Dirichlet form, so the operators $\rme^{-H_0t}$ are
positivity preserving for all $t\geq 0$. The identity $H_01=0$
implies that $\rme^{-H_0 t}1=1$ for all $t\geq 0$, so $\rme^{-H_0
t}$ is a symmetric Markov semigroup.

\begin{lemma} The operator $H_0$ on $L^2(X,\rmd x)$ satisfies (H1), (H2) and (H4).
\end{lemma}

\begin{proof}
If we disconnect $X$ by imposing Neumann boundary conditions
independently at the end of each edge, then we obtain a new operator
$H_1$ associated a quadratic form $Q_1$; this has the same formula
as $Q_0$, but a larger domain, consisting of all functions $f\in
L^2(X,\rmd x)$ such that the restriction of $f$ to any edge $e$ lies
in $W^{1,2}(e)$. The operator $H_1$ acts independently in each
$L^2(e,\rmd x)$ and its heat kernel in $e$ is of the form
\[
K_e(t,x,y)=\frac{1}{a}+\frac{2}{a}\sum_{n=1}^\infty\rme^{-\pi^2n^2
t/a^2} \cos(\pi n x/a)\cos(\pi n y/a),
\]
where we parametrize  $e$ by $(0,a)$. One readily sees that each
$K_e$ is continuous and that
\[
|K_e(t,x,y)|\leq c_1 t^{-1/2}
\]
for some $c_1>0$, all $x,y\in e$ and all $0<t<1$. Moreover
$K_e(t,x,y)\geq 0$ because $Q_1$ is a Dirichlet form. It follows
from these observations that the various equivalent conditions of
Proposition~\ref{logsob} hold for $Q_1$ with $d=1$. Since $Q_0$ is a
restriction of $Q_1$, Proposition~\ref{logsob} implies that
\[
0\leq K_0(t,x,y)\leq c_2 t^{-1/2}
\]
for some $c_2>0$, all $x,y\in X$ and all $0<t<1$. This completes the
proof of (H2).

To prove (H1) we note that if $\{ \phi_n\}_{n=1}^\infty$ is an
orthonormal basis of eigenfunctions of $H_0$ and $\lam_n$ are the
corresponding eigenvalues, then
\[
\phi_n\in \Dom(H_0)\subseteq \Dom(H_0^{1/2})=W^{1,2}(X)\subset C(X).
\]
Since the series
\[
K_0(t,x,y)=\sum_{n=1}^\infty \rme^{-\lam_n t} \phi_n(x)\phi_n(y)
\]
converges uniformly on $[\alp,\infty)\times X\times X$ for every
$\alp >0$ by \cite[Theorem~2.1.4]{HKST}, we deduce that $K_0$ is
continuous on $(0,1)\times X\times X$.

The proof of (H4) depends on the observation that $H_0\phi=0$ if and
only if $\phi\in W^{1,2}(X)\subset C(X)$ and
\[
0=Q_0(\phi)=\int_X|\phi^\pr(x)|^2\, \rmd x.
\]
This implies that $\phi$ is constant. Therefore $0$ is an eigenvalue
of multiplicity $1$.

\end{proof}

Our final task is to prove (H3).

\begin{lemma}
Let $K_a(t,x,y)$ be the heat kernel of the operator
$-\frac{\rmd^2}{\rmd x^2}$ acting in $L^2(-a,a)$ subject to
Dirichlet boundary conditions at $\pm a$. Then
\begin{equation}
0\leq K_a(t,x,y)\leq K_\infty(t,x,y)=(4\pi
t)^{-1/2}\rme^{-|x-y|^2/(4t)}\label{interval1}
\end{equation}
for all $t>0$ and $x,y\in (-a,a)$. Moreover
\begin{equation}
1\geq \int_{-a}^a K_a(t,0,x)\, \rmd x \geq
1-4\rme^{-a^2/(8t)}\label{interval2}
\end{equation}
and
\begin{equation}
(4\pi t)^{-1/2}\geq K_a(t,0,0)\geq (4\pi t)^{-1/2}\left(
1-15\rme^{-a^2/(4t)}\right)\label{interval3}
\end{equation}
for all $t>0$.
\end{lemma}

\begin{proof}
The inequality (\ref{interval1}) follows directly from the
monotonicity of the Dirichlet heat kernel as a function of the
region. Sharper versions of the inequalities (\ref{interval2}) and (\ref{interval3}) may be proved by applying the Poisson summation formula to the explicit eigenfunction expansion of $K_a$, \cite{vandenberg}. An alternative proof of (\ref{interval2}) based on the properties of the underlying Brownian motion in given in \cite[Lemma 6.5]{DvdB}.

One may prove (\ref{interval3}) from (\ref{interval2}) as follows.
We define $f,g:\R\to (0,\infty)$ by
\begin{eqnarray*}
f(x)&=& \begin{choices}
K_a(t,0,x)&\mbox{ if $|x|\leq a$,}\\
0&\mbox{ otherwise,}
\end{choices}
\\
g(x)&=& K_\infty(t,0,x)=(4\pi t)^{-1/2}\rme^{-x^2/(4t)}.
\end{eqnarray*}
so that $0\leq f(x)\leq g(x)\leq (4\pi t)^{-1/2}$ for all $x\in \R$
and $\int_\R g(x)\, \rmd x =1$. Therefore
\begin{eqnarray*}
0&\leq& (8\pi t)^{-1/2}-K_a(2t,0,0)\\
&=& K_\infty(2t,0,0)-K_a(2t,0,0)\\
&=& \int_R\{ g(x)^2-f(x)^2 \} \, \rmd x\\
&\leq& \int_R \{ g(x)-f(x)\} 2g(x) \, \rmd x\\
&\leq & (\pi t)^{-1/2} \int_R  \{ g(x)-f(x)\} \, \rmd x\\
&=& (\pi t)^{-1/2}\left( 1-\int_{-a}^a f(x)\, \rmd x\right)\\
&\leq& (\pi t)^{-1/2} 4\rme^{-a^2/(8t)}.
\end{eqnarray*}
by (H2). We finally obtain (\ref{interval3}) upon replacing $t$ by
$t/2$.
\end{proof}

We prove (H3) by using the principle of `not feeling the boundary', \cite{hsu}.

\begin{theorem}\label{NFTB}
The operator $H_0$ on $L^2(X,\rmd x)$ satisfies (H3), the
exceptional set $N$ being the set of all vertices on $X$.
\end{theorem}

\begin{proof}
This repeats the argument used to prove (\ref{interval3}). We assume
that $z\in X$ is not a vertex and that $a>0$ is its distance from
the closest vertex. We then let $K_a$ denote the Dirichlet heat
kernel for the interval $I$ with centre $z$ and length $2a$. Our
task is to compare the heat kernel $K$ of $X$ with $K_a$. We use the
following facts.
\[
0\leq K_a(t,x,y)\leq K_0(t,x,y)
\]
for all $t>0$ and $x,y\in X$, where we put $K_a(t,x,y)=0$ if $x$ or
$y$ does not lie in the interval $I$. In addition $0\leq K_0(t,x,y)
\leq ct^{-1/2}$ for all $0<t<1$ and all $x,y\in X$. Finally
\[
\int_XK_0(t,x,y)\, \rmd y=1
\]
for all $t>0$ and $x\in X$.

Let $f(x)=K_a(t,z,x)$ and $g(x)=K_0(t,z,x)$ so that $0\leq f(x)\leq
g(x)$ for all $x\in X$. We have
\begin{eqnarray*}
0&\leq & \int_X \{ g(x)-f(x) \} \, \rmd x = 1-\int_I K_a(t,z,x)\, \rmd x\\
&\leq & 4\rme^{-a^2/(8t)}
\end{eqnarray*}
by (\ref{interval2}). Therefore
\begin{eqnarray*}
K_0(2t,z,z)-K_a(2t,z,z)&= & \int_X \{ g(x)^2-f(x)^2 \} \, \rmd x\\
&\leq & \int_X \{ g(x)-f(x)\} 2 g(x)\, \rmd x\\
&\leq & c_1t^{-1/2} \rme^{-a^2/(8t)}.
\end{eqnarray*}
The theorem follows by combining this with (\ref{interval3}).
\end{proof}

Department of Mathematics\\
King's College London\\
Strand\\
London WC2R 2LS\\
UK

E.Brian.Davies@kcl.ac.uk

\end{document}